\documentclass[reqno]{amsart}

\usepackage{amsmath,amssymb,amsthm,amsfonts,mathrsfs}
\usepackage{fullpage}
\usepackage{xcolor}
\usepackage{graphicx}
\usepackage{listings}
\usepackage{bbm}

\definecolor{codegreen}{rgb}{0,0.6,0}
\definecolor{codegray}{rgb}{0.5,0.5,0.5}
\definecolor{codepurple}{rgb}{0.58,0,0.82}
\definecolor{backcolour}{rgb}{0.95,0.95,0.92}

\lstdefinestyle{mystyle}{
  backgroundcolor=\color{backcolour},   commentstyle=\color{codegreen},
  keywordstyle=\color{magenta},
  numberstyle=\tiny\color{codegray},
  stringstyle=\color{codepurple},
  basicstyle=\ttfamily\footnotesize,
  breakatwhitespace=false,         
  breaklines=true,                 
  captionpos=b,                    
  keepspaces=true,                 
  numbers=left,                    
  numbersep=5pt,                  
  showspaces=false,                
  showstringspaces=false,
  showtabs=false,                  
  tabsize=2
}

\usepackage{listings}
\usepackage[colorlinks=true,
linkcolor=blue,
anchorcolor=blue,
citecolor=red
]{hyperref}
\allowdisplaybreaks 
\newtheorem{theorem}{Theorem}[section]
\newtheorem{lemma}[theorem]{Lemma}

\newtheorem*{conjecture*}{Conjecture}
\theoremstyle{definition}

\theoremstyle{remark}

\newtheorem*{remark*}{remark}

\usepackage{colonequals}

\author{Runbo Li}
\address{International Curriculum Center, The High School Affiliated to Renmin University of China, Beijing, China}
\email{runbo.li.carey@gmail.com}

\makeatletter
\@namedef{subjclassname@2020}{\textup{}2020 Mathematics Subject Classification}
\makeatother

\title[]{On prime-producing sieves and distribution of $\alpha p - \beta$ mod $1$}
\subjclass[2020]{11N35, 11N36} 
\keywords{prime, sieve methods, asymptotic formula}

\begin{document}
	
\begin{abstract}
The author proves that there are infinitely many primes $p$ such that $\| \alpha p - \beta \| < p^{-\frac{28}{87}}$, where $\alpha$ is an irrational number and $\beta$ is a real number. This sharpens a result of Jia (2000) and provides a new triple $(\gamma, \theta, \nu)=(\frac{59}{87}, \frac{28}{87}, \frac{1}{29})$ that can produce special primes in Ford and Maynard's work on prime-producing sieves. Our minimum amount of Type-II information required ($\nu = \frac{1}{29}$) is less than any previous work on this topic using only traditional Type-I and Type-II information.
\end{abstract}

\maketitle

\tableofcontents

\section{Introduction}
Let $\alpha$ be an irrational number and $\| y \|$ denote the smallest distance from $y$ to integers. Earlier work on this problem was done by Vinogradov \cite{Vinogradov} in 1954, who showed that for any real number $\beta$, there are infinitely many prime numbers $p$ such that if $\tau = \frac{1}{5}-\varepsilon$, then
\begin{equation}
\| \alpha p - \beta \| < p^{-\tau}. 
\end{equation}
In 1977, Vaughan \cite{Vaughan} got $\tau = \frac{1}{4}-\varepsilon$ using his identity. In 1983, Harman \cite{Harman1} introduced a new sieve method to this topic and got $\tau = \frac{3}{10}$. Jia \cite{Jia413} improved it to $\tau = \frac{4}{13}$ in 1993. In 1996, Harman \cite{Harman2} further improved it to $\tau = \frac{7}{22}$ by applying a new technique (the variable role-reversal) in his sieve. In 2000, Jia \cite{Jia928} got $\theta = \frac{9}{28}$. It is worth to mention that Balog \cite{Balog1986} also got the same result in 1986 under the condition that $\| \alpha n \| < n^{-\frac{43}{31}-\varepsilon}$ holds for infinitely many integers $n$. If we only focus on the special case $\beta = 0$, then even better exponents $\frac{16}{49}$ and $\frac{1}{3}-\varepsilon$ were obtained by Heath-Brown and Jia \cite{HeathBrownJia} and Matomäki \cite{Matomaki} respectively. In a personal communication, Matomäki mentioned that Maynard has got some $\tau > \frac{1}{3}$. Note that the Riemann Hypothesis implies that (1) holds for $\tau = \frac{1}{3}-\varepsilon$. In this paper, we show that (1) holds for $\tau = \frac{28}{87}$.
\begin{theorem}\label{t1}
Suppose that $\alpha$ is an irrational number, then for any real number $\beta$, there are infinitely many prime numbers $p$ such that
$$
\| \alpha p - \beta \| < p^{-\frac{28}{87}}. 
$$
\end{theorem}

A direct corollary of our Theorem~\ref{t1} is the distribution of $p^{\theta} - \beta$ mod $1$ for some $\theta < 1$.
\begin{theorem}\label{t2}
For $\frac{31}{87} \leqslant \theta < 1$ and any real number $\beta$, there are infinitely many prime numbers $p$ such that
$$
\| p^{\theta} - \beta \| < p^{-\frac{1-\theta}{2}+\varepsilon}. 
$$
\end{theorem}

Another corollary of our Theorem~\ref{t1} is the following result focusing on Diophantine approximation with Gaussian primes, which improves Harman's exponent $\frac{7}{22}$ \cite{HarmanGaussian722}.
\begin{theorem}\label{t3}
Let $0 \leqslant \omega_1 < \omega_2 \leqslant 2\pi$. Then, given $\alpha \in \mathbb{C} \backslash \mathbb{Q}[i], \beta \in \mathbb{C}$, there are infinitely many Gaussian primes $\frak{p}$ such that
$$
\| \alpha \frak{p} - \beta \| < |\frak{p}|^{-\frac{28}{87}}, \quad \omega_1 \leqslant \operatorname{arg} \frak{p} \leqslant \omega_2.
$$
\end{theorem}

In 2024, Ford and Maynard \cite{FordMaynard} considered a series of general problems on prime-producing sieves. For those sieves that only use Type-I and Type-II information inputs, they defined a triple $(\gamma, \theta, \nu)$ and considered various values of $\gamma, \theta \text{ and } \nu$ that can produce primes. In their notation, our Theorem~\ref{t1} implies the following result:
\begin{theorem}\label{t4}
The triple $(\gamma, \theta, \nu) = (\frac{59}{87}, \frac{28}{87}, \frac{1}{29})$ can produce primes with the property $\| \alpha p - \beta \| < p^{-\frac{28}{87}}$.
\end{theorem}

Throughout this paper, we suppose that $\frac{a}{q}$ is a convergent to the continued fraction for $\alpha$ and $\varepsilon$ is a sufficiently small positive constant. The letter $p$, with or without subscript, is reserved for prime numbers. Let $\tau = \frac{28}{87}$, $x = q^{\frac{2}{1+\tau}}$ and $\delta = (2x)^{-\tau}$.

\section{Asymptotic formulas}
Now we follow the discussion in \cite{Jia928}. Let $p_{j}=x^{t_j}$ and put
$$
\mathcal{B}=\{n: x<n \leqslant 2 x\}, \quad \mathcal{A}=\{n: x<n \leqslant 2 x,\ \| \alpha n - \beta \| < \delta \},
$$
$$
\mathcal{A}_d=\{a: ad \in \mathcal{A}\}, \quad P(z)=\prod_{p<z} p, \quad S(\mathcal{A}, z)=\sum_{\substack{a \in \mathcal{A} \\ (a, P(z))=1}} 1.
$$
Then we only need to show that $S\left(\mathcal{A},(2 x)^{\frac{1}{2}}\right) >0$. Our aim is to show that the sparser set $\mathcal{A}$ contains the expected proportion of primes compared to the bigger set $\mathcal{B}$, which requires us to decompose $S\left(\mathcal{A}, (2 v)^{\frac{1}{2}}\right)$ and prove asymptotic formulas of the form
\begin{equation}
S\left(\mathcal{A}, z\right) = 2 \delta (1+o(1)) S\left(\mathcal{B}, z\right)
\end{equation}
for some parts of it, and drop the other positive parts.

Here are some known asymptotic formulas for types of sieve functions, which were first proved by Harman \cite{Harman1} using traditional Type-I and Type-II arithmetic information.
\begin{lemma}\label{l21} (Type-I)
Suppose that $M \ll x^{\frac{59}{87}}$ and $a(m) = O(1)$. Then we have
$$
\sum_{m \sim M} a(m) S\left(\mathcal{A}_{m}, x^{\frac{1}{29}}\right) = 2 \delta (1+o(1)) \sum_{m \sim M} a(m) S\left(\mathcal{B}_{m}, x^{\frac{1}{29}}\right).
$$
\end{lemma}

\begin{proof}
The proof is similar to that of [\cite{Harman1}, Lemma 7].
\end{proof}

\begin{lemma}\label{l22} (Type-II)
Suppose that $x^{\frac{28}{87}} \ll M \ll x^{\frac{31}{87}}$ or $x^{\frac{56}{87}} \ll M \ll x^{\frac{59}{87}}$ and that $a(m), b(n) = O(1)$. Then we have
$$
\sum_{m \sim M} \sum_{n} a(m) b(n) S\left(\mathcal{A}_{m n}, v(m,n)\right) = 2 \delta (1+o(1)) \sum_{m \sim M} \sum_{n} a(m) b(n) S\left(\mathcal{B}_{m n}, v(m,n)\right).
$$
\end{lemma}

\begin{proof}
The proof is similar to that of [\cite{Harman1}, Lemma 6].
\end{proof}

Many authors, such as Heath-Brown and Jia \cite{HeathBrownJia} and Matomäki \cite{Matomaki}, have found more arithmetic information by using estimations of Kloosterman sums. However, their methods usually require $\beta = 0$, which is not applicable on the present paper.

\section{The final decomposition}
Let $\omega(u)$ denote the Buchstab function determined by the following differential-difference equation
\begin{align*}
\begin{cases}
\omega(u)=\frac{1}{u}, & \quad 1 \leqslant u \leqslant 2, \\
(u \omega(u))^{\prime}= \omega(u-1), & \quad u \geqslant 2 .
\end{cases}
\end{align*}
Moreover, we have the upper and lower bounds for $\omega(u)$:
\begin{align*}
\omega(u) \geqslant \omega_{0}(u) =
\begin{cases}
\frac{1}{u}, & \quad 1 \leqslant u < 2, \\
\frac{1+\log(u-1)}{u}, & \quad 2 \leqslant u < 3, \\
\frac{1+\log(u-1)}{u} + \frac{1}{u} \int_{2}^{u-1}\frac{\log(t-1)}{t} d t \geqslant 0.5607, & \quad 3 \leqslant u < 4, \\
0.5612, & \quad u \geqslant 4, \\
\end{cases}
\end{align*}
\begin{align*}
\omega(u) \leqslant \omega_{1}(u) =
\begin{cases}
\frac{1}{u}, & \quad 1 \leqslant u < 2, \\
\frac{1+\log(u-1)}{u}, & \quad 2 \leqslant u < 3, \\
\frac{1+\log(u-1)}{u} + \frac{1}{u} \int_{2}^{u-1}\frac{\log(t-1)}{t} d t \leqslant 0.5644, & \quad 3 \leqslant u < 4, \\
0.5617, & \quad u \geqslant 4. \\
\end{cases}
\end{align*}
We shall use $\omega_0(u)$ and $\omega_1(u)$ to give numerical bounds for some sieve functions discussed below. We shall also use the simple upper bound $\omega(u) \leqslant \max(\frac{1}{u}, 0.5672)$ (see Lemma 8(iii) of \cite{JiaPSV}) to estimate high-dimensional integrals.

Before decomposing, we define the asymptotic region $I$ as
\begin{align}
\nonumber I (m, n) :=&\ \left\{ \frac{28}{87} \leqslant m \leqslant \frac{31}{87} \text{ or } \frac{56}{87} \leqslant m \leqslant \frac{59}{87} \text{ or } \frac{28}{87} \leqslant n \leqslant \frac{31}{87} \text{ or } \frac{56}{87} \leqslant n \leqslant \frac{59}{87} \right. \\
\nonumber & \left. \qquad \text{ or } \frac{28}{87} \leqslant m + n \leqslant \frac{31}{87} \text{ or } \frac{56}{87} \leqslant m + n \leqslant \frac{59}{87} \right\}.
\end{align}

By Buchstab's identity, we have
\begin{align}
\nonumber S\left(\mathcal{A}, (2 x)^{\frac{1}{2}}\right) =&\ S\left(\mathcal{A}, x^{\frac{1}{29}}\right) - \sum_{\frac{1}{29} \leqslant t_1 < \frac{1}{2}} S\left(\mathcal{A}_{p_1}, x^{\frac{1}{29}}\right) + \sum_{\substack{\frac{1}{29} \leqslant t_1 < \frac{1}{2} \\ \frac{1}{29} \leqslant t_2 < \min \left(t_1, \frac{1}{2}(1 - t_1) \right) }} S\left(\mathcal{A}_{p_1 p_2}, p_2\right) \\
=&\ S_1 - S_2 + S_3.
\end{align}
By Lemma~\ref{l21}, we can give asymptotic formulas for $S_1$ and $S_2$. Before estimating $S_{3}$, we first split it into six parts:
\begin{align}
\nonumber S_3 =&\ \sum_{\substack{\frac{1}{29} \leqslant t_1 < \frac{1}{2} \\ \frac{1}{29} \leqslant t_2 < \min \left(t_1, \frac{1}{2}(1 - t_1) \right) }} S\left(\mathcal{A}_{p_1 p_2}, p_2\right) \\
\nonumber =&\ \sum_{\substack{\frac{1}{29} \leqslant t_1 < \frac{1}{2} \\ \frac{1}{29} \leqslant t_2 < \min \left(t_1, \frac{1}{2}(1 - t_1) \right) \\ (t_1, t_2) \in I }} S\left(\mathcal{A}_{p_1 p_2}, p_2\right) + \sum_{\substack{\frac{1}{29} \leqslant t_1 < \frac{1}{2} \\ \frac{1}{29} \leqslant t_2 < \min \left(t_1, \frac{1}{2}(1 - t_1) \right)  \\ (t_1, t_2) \notin I \\ t_1 < \frac{28}{87},\ t_1 + 2 t_2 > \frac{59}{87} \text{ or } t_1 + t_2 > \frac{59}{87} }} S\left(\mathcal{A}_{p_1 p_2}, p_2\right)  \\
\nonumber &+ \sum_{\substack{\frac{1}{29} \leqslant t_1 < \frac{1}{2} \\ \frac{1}{29} \leqslant t_2 < \min \left(t_1, \frac{1}{2}(1 - t_1) \right)  \\ (t_1, t_2) \notin I \\ t_1 \leqslant \frac{31}{87},\ t_1 + 2 t_2 \leqslant \frac{59}{87} }} S\left(\mathcal{A}_{p_1 p_2}, p_2\right) + \sum_{\substack{\frac{1}{29} \leqslant t_1 < \frac{1}{2} \\ \frac{1}{29} \leqslant t_2 < \min \left(t_1, \frac{1}{2}(1 - t_1) \right)  \\ (t_1, t_2) \notin I \\ t_1 > \frac{31}{87},\ t_1 + 2 t_2 \leqslant \frac{59}{87} }} S\left(\mathcal{A}_{p_1 p_2}, p_2\right) \\
\nonumber &+ \sum_{\substack{\frac{1}{29} \leqslant t_1 < \frac{1}{2} \\ \frac{1}{29} \leqslant t_2 < \min \left(t_1, \frac{1}{2}(1 - t_1) \right)  \\ (t_1, t_2) \notin I \\ t_1 > \frac{31}{87},\ t_1 + 2 t_2 > \frac{59}{87} \\ t_2 \leqslant \max\left(\frac{1-t_1}{3}, \frac{t_1}{2}\right)}} S\left(\mathcal{A}_{p_1 p_2}, p_2\right) + \sum_{\substack{\frac{1}{29} \leqslant t_1 < \frac{1}{2} \\ \frac{1}{29} \leqslant t_2 < \min \left(t_1, \frac{1}{2}(1 - t_1) \right)  \\ (t_1, t_2) \notin I \\ t_1 > \frac{31}{87},\ t_1 + 2 t_2 > \frac{59}{87} \\ t_2 > \max\left(\frac{1-t_1}{3}, \frac{t_1}{2}\right)}} S\left(\mathcal{A}_{p_1 p_2}, p_2\right) \\
\nonumber =&\ \sum_{\substack{\frac{1}{29} \leqslant t_1 < \frac{1}{2} \\ \frac{1}{29} \leqslant t_2 < \min \left(t_1, \frac{1}{2}(1 - t_1) \right)  \\ (t_1, t_2) \in I}} S\left(\mathcal{A}_{p_1 p_2}, p_2\right) + \sum_{\substack{\frac{1}{29} \leqslant t_1 < \frac{1}{2} \\ \frac{1}{29} \leqslant t_2 < \min \left(t_1, \frac{1}{2}(1 - t_1) \right)  \\ (t_1, t_2) \in J_1}} S\left(\mathcal{A}_{p_1 p_2}, p_2\right) \\
\nonumber &+ \sum_{\substack{\frac{1}{29} \leqslant t_1 < \frac{1}{2} \\ \frac{1}{29} \leqslant t_2 < \min \left(t_1, \frac{1}{2}(1 - t_1) \right)  \\ (t_1, t_2) \in J_2}} S\left(\mathcal{A}_{p_1 p_2}, p_2\right) + \sum_{\substack{\frac{1}{29} \leqslant t_1 < \frac{1}{2} \\ \frac{1}{29} \leqslant t_2 < \min \left(t_1, \frac{1}{2}(1 - t_1) \right)  \\ (t_1, t_2) \in J_3}} S\left(\mathcal{A}_{p_1 p_2}, p_2\right) \\
\nonumber &+ \sum_{\substack{\frac{1}{29} \leqslant t_1 < \frac{1}{2} \\ \frac{1}{29} \leqslant t_2 < \min \left(t_1, \frac{1}{2}(1 - t_1) \right)  \\ (t_1, t_2) \in J_4}} S\left(\mathcal{A}_{p_1 p_2}, p_2\right) + \sum_{\substack{\frac{1}{29} \leqslant t_1 < \frac{1}{2} \\ \frac{1}{29} \leqslant t_2 < \min \left(t_1, \frac{1}{2}(1 - t_1) \right)  \\ (t_1, t_2) \in J_5}} S\left(\mathcal{A}_{p_1 p_2}, p_2\right) \\
=&\ S_{31} + S_{32} + S_{33} + S_{34} + S_{35} + S_{36}.
\end{align}

$S_{31}$ has an asymptotic formula. For $S_{32}$, we cannot decompose further but have to discard the whole region giving the loss
\begin{equation}
L_{32} := \int_{\frac{1}{29}}^{\frac{1}{2}} \int_{\frac{1}{29}}^{\min\left(t_1, \frac{1-t_1}{2}\right)} \mathbbm{1}_{(t_1, t_2) \in J_1} \frac{\omega\left(\frac{1 - t_1 - t_2}{t_2}\right)}{t_1 t_2^2} d t_2 d t_1 < 0.397685.
\end{equation}
For $S_{33}$ we can use Buchstab’s identity to get
\begin{align}
\nonumber S_{33} =&\ \sum_{\substack{\frac{1}{29} \leqslant t_1 < \frac{1}{2} \\ \frac{1}{29} \leqslant t_2 < \min \left(t_1, \frac{1}{2}(1 - t_1) \right)  \\ (t_1, t_2) \in J_2}} S\left(\mathcal{A}_{p_1 p_2}, p_2\right) \\
\nonumber =&\ \sum_{\substack{\frac{1}{29} \leqslant t_1 < \frac{1}{2} \\ \frac{1}{29} \leqslant t_2 < \min \left(t_1, \frac{1}{2}(1 - t_1) \right)  \\ (t_1, t_2) \in J_2}} S\left(\mathcal{A}_{p_1 p_2}, x^{\frac{1}{29}}\right) - \sum_{\substack{\frac{1}{29} \leqslant t_1 < \frac{1}{2} \\ \frac{1}{29} \leqslant t_2 < \min \left(t_1, \frac{1}{2}(1 - t_1) \right)  \\ (t_1, t_2) \in J_2 \\ \frac{1}{29} \leqslant t_3 < \min \left(t_2, \frac{1}{2}(1 - t_1 - t_2) \right) }} S\left(\mathcal{A}_{p_1 p_2 p_3}, x^{\frac{1}{29}}\right) \\
\nonumber &+ \sum_{\substack{\frac{1}{29} \leqslant t_1 < \frac{1}{2} \\ \frac{1}{29} \leqslant t_2 < \min \left(t_1, \frac{1}{2}(1 - t_1) \right)  \\ (t_1, t_2) \in J_2 \\ \frac{1}{29} \leqslant t_3 < \min \left(t_2, \frac{1}{2}(1 - t_1 - t_2) \right) \\ \frac{1}{29} \leqslant t_4 < \min \left(t_3, \frac{1}{2}(1 - t_1 - t_2 - t_3) \right) \\ (t_1, t_2, t_3, t_4) \text{ can be partitioned into } (m, n) \in I }} S\left(\mathcal{A}_{p_1 p_2 p_3 p_4}, p_4 \right) \\
\nonumber &+ \sum_{\substack{\frac{1}{29} \leqslant t_1 < \frac{1}{2} \\ \frac{1}{29} \leqslant t_2 < \min \left(t_1, \frac{1}{2}(1 - t_1) \right)  \\ (t_1, t_2) \in J_2 \\ \frac{1}{29} \leqslant t_3 < \min \left(t_2, \frac{1}{2}(1 - t_1 - t_2) \right) \\ \frac{1}{29} \leqslant t_4 < \min \left(t_3, \frac{1}{2}(1 - t_1 - t_2 - t_3) \right) \\ (t_1, t_2, t_3, t_4) \text{ cannot be partitioned into } (m, n) \in I }} S\left(\mathcal{A}_{p_1 p_2 p_3 p_4}, p_4 \right) \\
=&\ S_{331} - S_{332} + S_{333} + S_{334}.
\end{align}
We have asymptotic formulas for $S_{331}$--$S_{333}$. For the remaining $S_{334}$, we have two ways to get more possible savings: One way is to use Buchstab's identity twice more for some parts if we have $t_1 + t_2 + t_3 + 2 t_4 \leqslant \frac{59}{87}$. Another way is to use Buchstab's identity in reverse to make almost-primes visible. The details of further decompositions are similar to those in \cite{LRB052}. Combining the cases above we get a loss from $S_{33}$ of
\begin{align}
\nonumber & \left( \int_{(t_1, t_2, t_3, t_4) \in J_{331}} \frac{\omega \left(\frac{1 - t_1 - t_2 - t_3 - t_4}{t_4}\right)}{t_1 t_2 t_3 t_4^2} d t_4 d t_3 d t_2 d t_1 \right) \\
\nonumber -& \left( \int_{(t_1, t_2, t_3, t_4, t_5) \in J_{332}} \frac{\omega \left(\frac{1 - t_1 - t_2 - t_3 - t_4 - t_5}{t_5}\right)}{t_1 t_2 t_3 t_4 t_5^2} d t_5 d t_4 d t_3 d t_2 d t_1 \right) \\
\nonumber +& \left( \int_{(t_1, t_2, t_3, t_4, t_5, t_6) \in J_{333}} \frac{\omega_1 \left(\frac{1 - t_1 - t_2 - t_3 - t_4 - t_5 - t_6}{t_6}\right)}{t_1 t_2 t_3 t_4 t_5 t_6^2} d t_6 d t_5 d t_4 d t_3 d t_2 d t_1 \right) \\
\nonumber +& \left( \int_{(t_1, t_2, t_3, t_4, t_5, t_6, t_7, t_8) \in J_{334}} \frac{\max \left(\frac{t_8}{1 - t_1 - t_2 - t_3 - t_4 - t_5 - t_6 - t_7 - t_8}, 0.5672\right)}{t_1 t_2 t_3 t_4 t_5 t_6 t_7 t_8^2} d t_8 d t_7 d t_6 d t_5 d t_4 d t_3 d t_2 d t_1 \right) \\
\nonumber =&\ L_{331} - L_{332} + L_{333} + L_{334}, \\
\nonumber <&\ 0.111391 - 0.021501 + 0.001491 + 0.000002 \\
=&\ 0.091383,
\end{align}
where
\begin{align}
\nonumber J_{331}(t_1, t_2, t_3, t_4) :=&\ \left\{ (t_1, t_2) \in J_{2}, \right. \\
\nonumber & \quad \frac{1}{29} \leqslant t_3 < \min\left(t_2, \frac{1}{2}(1-t_1-t_2)\right), \\ 
\nonumber & \quad \frac{1}{29} \leqslant t_4 < \min\left(t_3, \frac{1}{2}(1-t_1-t_2-t_3) \right), \\
\nonumber & \quad (t_1, t_2, t_3, t_4) \text{ cannot be partitioned into } (m, n) \in I, \\
\nonumber & \quad t_1 + t_2 + t_3 + 2 t_4 > \frac{59}{87}, \\
\nonumber & \left. \quad \frac{1}{29} \leqslant t_1 < \frac{1}{2},\ \frac{1}{29} \leqslant t_2 < \min\left(t_1, \frac{1}{2}(1-t_1) \right) \right\}, \\
\nonumber J_{332}(t_1, t_2, t_3, t_4, t_5) :=&\ \left\{ (t_1, t_2) \in J_{2}, \right. \\
\nonumber & \quad \frac{1}{29} \leqslant t_3 < \min\left(t_2, \frac{1}{2}(1-t_1-t_2)\right), \\ 
\nonumber & \quad \frac{1}{29} \leqslant t_4 < \min\left(t_3, \frac{1}{2}(1-t_1-t_2-t_3) \right), \\
\nonumber & \quad (t_1, t_2, t_3, t_4) \text{ cannot be partitioned into } (m, n) \in I, \\
\nonumber & \quad t_1 + t_2 + t_3 + 2 t_4 > \frac{59}{87}, \\
\nonumber & \quad t_4 < t_5 < \frac{1}{2}(1-t_1-t_2-t_3-t_4), \\
\nonumber & \quad (t_1, t_2, t_3, t_4, t_5) \text{ can be partitioned into } (m, n) \in I, \\
\nonumber & \left. \quad \frac{1}{29} \leqslant t_1 < \frac{1}{2},\ \frac{1}{29} \leqslant t_2 < \min\left(t_1, \frac{1}{2}(1-t_1) \right) \right\}, \\
\nonumber J_{333}(t_1, t_2, t_3, t_4, t_5, t_6) :=&\ \left\{ (t_1, t_2) \in J_{2}, \right. \\
\nonumber & \quad \frac{1}{29} \leqslant t_3 < \min\left(t_2, \frac{1}{2}(1-t_1-t_2)\right), \\ 
\nonumber & \quad \frac{1}{29} \leqslant t_4 < \min\left(t_3, \frac{1}{2}(1-t_1-t_2-t_3) \right), \\
\nonumber & \quad (t_1, t_2, t_3, t_4) \text{ cannot be partitioned into } (m, n) \in I, \\
\nonumber & \quad t_1 + t_2 + t_3 + 2 t_4 \leqslant \frac{59}{87}, \\
\nonumber & \quad \frac{1}{29} \leqslant t_5 < \min\left(t_4, \frac{1}{2}(1-t_1-t_2-t_3-t_4) \right), \\
\nonumber & \quad \frac{1}{29} \leqslant t_6 < \min\left(t_5, \frac{1}{2}(1-t_1-t_2-t_3-t_4-t_5) \right),\\
\nonumber & \quad (t_1, t_2, t_3, t_4, t_5, t_6) \text{ cannot be partitioned into } (m, n) \in I, \\
\nonumber & \quad t_1 + t_2 + t_3 + t_4 + t_5 + 2 t_6 > \frac{59}{87}, \\
\nonumber & \left. \quad \frac{1}{29} \leqslant t_1 < \frac{1}{2},\ \frac{1}{29} \leqslant t_2 < \min\left(t_1, \frac{1}{2}(1-t_1) \right) \right\}, \\
\nonumber J_{334}(t_1, t_2, t_3, t_4, t_5, t_6, t_7, t_8) :=&\ \left\{ (t_1, t_2) \in J_{2}, \right. \\
\nonumber & \quad \frac{1}{29} \leqslant t_3 < \min\left(t_2, \frac{1}{2}(1-t_1-t_2)\right), \\ 
\nonumber & \quad \frac{1}{29} \leqslant t_4 < \min\left(t_3, \frac{1}{2}(1-t_1-t_2-t_3) \right), \\
\nonumber & \quad (t_1, t_2, t_3, t_4) \text{ cannot be partitioned into } (m, n) \in I, \\
\nonumber & \quad t_1 + t_2 + t_3 + 2 t_4 \leqslant \frac{59}{87}, \\
\nonumber & \quad \frac{1}{29} \leqslant t_5 < \min\left(t_4, \frac{1}{2}(1-t_1-t_2-t_3-t_4) \right), \\
\nonumber & \quad \frac{1}{29} \leqslant t_6 < \min\left(t_5, \frac{1}{2}(1-t_1-t_2-t_3-t_4-t_5) \right),\\
\nonumber & \quad (t_1, t_2, t_3, t_4, t_5, t_6) \text{ cannot be partitioned into } (m, n) \in I, \\
\nonumber & \quad t_1 + t_2 + t_3 + t_4 + t_5 + 2 t_6 \leqslant \frac{59}{87}, \\
\nonumber & \quad \frac{1}{29} \leqslant t_7 < \min\left(t_5, \frac{1}{2}(1-t_1-t_2-t_3-t_4-t_5-t_6) \right),\\
\nonumber & \quad \frac{1}{29} \leqslant t_8 < \min\left(t_5, \frac{1}{2}(1-t_1-t_2-t_3-t_4-t_5-t_6-t_7) \right),\\
\nonumber & \quad (t_1, t_2, t_3, t_4, t_5, t_6, t_7, t_8) \text{ cannot be partitioned into } (m, n) \in I, \\
\nonumber & \left. \quad \frac{1}{29} \leqslant t_1 < \frac{1}{2},\ \frac{1}{29} \leqslant t_2 < \min\left(t_1, \frac{1}{2}(1-t_1) \right) \right\}.
\end{align}

Next we shall decompose $S_{34}$. By the same process as the decomposition of $S_{33}$ above, we can reach a four-dimentional sum
$$
\sum_{\substack{\frac{1}{29} \leqslant t_1 < \frac{1}{2} \\ \frac{1}{29} \leqslant t_2 < \min \left(t_1, \frac{1}{2}(1 - t_1) \right)  \\ (t_1, t_2) \in J_2 \\ \frac{1}{29} \leqslant t_3 < \min \left(t_2, \frac{1}{2}(1 - t_1 - t_2) \right) \\ \frac{1}{29} \leqslant t_4 < \min \left(t_3, \frac{1}{2}(1 - t_1 - t_2 - t_3) \right) \\ (t_1, t_2, t_3, t_4) \text{ cannot be partitioned into } (m, n) \in I }} S\left(\mathcal{A}_{p_1 p_2 p_3 p_4}, p_4 \right).
$$
We divide it into three parts:
\begin{align}
\nonumber & \sum_{\substack{\frac{1}{29} \leqslant t_1 < \frac{1}{2} \\ \frac{1}{29} \leqslant t_2 < \min \left(t_1, \frac{1}{2}(1 - t_1) \right)  \\ (t_1, t_2) \in J_2 \\ \frac{1}{29} \leqslant t_3 < \min \left(t_2, \frac{1}{2}(1 - t_1 - t_2) \right) \\ \frac{1}{29} \leqslant t_4 < \min \left(t_3, \frac{1}{2}(1 - t_1 - t_2 - t_3) \right) \\ (t_1, t_2, t_3, t_4) \text{ cannot be partitioned into } (m, n) \in I }} S\left(\mathcal{A}_{p_1 p_2 p_3 p_4}, p_4 \right) \\
\nonumber =&\ \sum_{\substack{\frac{1}{29} \leqslant t_1 < \frac{1}{2} \\ \frac{1}{29} \leqslant t_2 < \min \left(t_1, \frac{1}{2}(1 - t_1) \right)  \\ (t_1, t_2) \in J_2 \\ \frac{1}{29} \leqslant t_3 < \min \left(t_2, \frac{1}{2}(1 - t_1 - t_2) \right) \\ \frac{1}{29} \leqslant t_4 < \min \left(t_3, \frac{1}{2}(1 - t_1 - t_2 - t_3) \right) \\ (t_1, t_2, t_3, t_4) \text{ cannot be partitioned into } (m, n) \in I \\ t_1 + t_2 + t_3 + t_4 > \frac{59}{87} }} S\left(\mathcal{A}_{p_1 p_2 p_3 p_4}, p_4 \right) \\
\nonumber &+ \sum_{\substack{\frac{1}{29} \leqslant t_1 < \frac{1}{2} \\ \frac{1}{29} \leqslant t_2 < \min \left(t_1, \frac{1}{2}(1 - t_1) \right)  \\ (t_1, t_2) \in J_2 \\ \frac{1}{29} \leqslant t_3 < \min \left(t_2, \frac{1}{2}(1 - t_1 - t_2) \right) \\ \frac{1}{29} \leqslant t_4 < \min \left(t_3, \frac{1}{2}(1 - t_1 - t_2 - t_3) \right) \\ (t_1, t_2, t_3, t_4) \text{ cannot be partitioned into } (m, n) \in I \\ t_1 + t_2 + t_3 + 2 t_4 \leqslant \frac{59}{87} }} S\left(\mathcal{A}_{p_1 p_2 p_3 p_4}, p_4 \right) \\
&+ \sum_{\substack{\frac{1}{29} \leqslant t_1 < \frac{1}{2} \\ \frac{1}{29} \leqslant t_2 < \min \left(t_1, \frac{1}{2}(1 - t_1) \right)  \\ (t_1, t_2) \in J_2 \\ \frac{1}{29} \leqslant t_3 < \min \left(t_2, \frac{1}{2}(1 - t_1 - t_2) \right) \\ \frac{1}{29} \leqslant t_4 < \min \left(t_3, \frac{1}{2}(1 - t_1 - t_2 - t_3) \right) \\ (t_1, t_2, t_3, t_4) \text{ cannot be partitioned into } (m, n) \in I \\ t_1 + t_2 + t_3 + 2 t_4 > \frac{59}{87} \\ t_1 + t_2 + t_3 + t_4 \leqslant \frac{59}{87} }} S\left(\mathcal{A}_{p_1 p_2 p_3 p_4}, p_4 \right).
\end{align}
We can only use Buchstab's identity in reverse to make more savings for the first sum on the right-hand side. For the second sum in (8), we can perform a straightforward decomposition to get a six-dimensional sum
\begin{equation}
\sum_{\substack{\frac{1}{29} \leqslant t_1 < \frac{1}{2} \\ \frac{1}{29} \leqslant t_2 < \min \left(t_1, \frac{1}{2}(1 - t_1) \right)  \\ (t_1, t_2) \in J_3 \\ \frac{1}{29} \leqslant t_3 < \min \left(t_2, \frac{1}{2}(1 - t_1 - t_2) \right) \\ \frac{1}{29} \leqslant t_4 < \min \left(t_3, \frac{1}{2}(1 - t_1 - t_2 - t_3) \right) \\ (t_1, t_2, t_3, t_4) \text{ cannot be partitioned into } (m, n) \in I \\ t_1 + t_2 + t_3 + 2 t_4 \leqslant \frac{59}{87} \\ \frac{1}{29} \leqslant t_5 < \min\left(t_4, \frac{1}{2}(1-t_1-t_2-t_3-t_4) \right) \\ \frac{1}{29} \leqslant t_6 < \min\left(t_5, \frac{1}{2}(1-t_1-t_2-t_3-t_4-t_5)\right) \\ (t_1, t_2, t_3, t_4, t_5, t_6) \text{ cannot be partitioned into } (m, n) \in I \\ t_1 + t_2 + t_3 + t_4 + t_5 + 2 t_6 > \frac{59}{87} }} S\left(\mathcal{A}_{p_1 p_2 p_3 p_4 p_5 p_6}, p_6 \right)
\end{equation}
and an eight-dimensional sum
\begin{equation}
\sum_{\substack{\frac{1}{29} \leqslant t_1 < \frac{1}{2} \\ \frac{1}{29} \leqslant t_2 < \min \left(t_1, \frac{1}{2}(1 - t_1) \right)  \\ (t_1, t_2) \in J_3 \\ \frac{1}{29} \leqslant t_3 < \min \left(t_2, \frac{1}{2}(1 - t_1 - t_2) \right) \\ \frac{1}{29} \leqslant t_4 < \min \left(t_3, \frac{1}{2}(1 - t_1 - t_2 - t_3) \right) \\ (t_1, t_2, t_3, t_4) \text{ cannot be partitioned into } (m, n) \in I \\ t_1 + t_2 + t_3 + 2 t_4 \leqslant \frac{59}{87} \\ \frac{1}{29} \leqslant t_5 < \min\left(t_4, \frac{1}{2}(1-t_1-t_2-t_3-t_4) \right) \\ \frac{1}{29} \leqslant t_6 < \min\left(t_5, \frac{1}{2}(1-t_1-t_2-t_3-t_4-t_5)\right) \\ (t_1, t_2, t_3, t_4, t_5, t_6) \text{ cannot be partitioned into } (m, n) \in I \\ t_1 + t_2 + t_3 + t_4 + t_5 + 2 t_6 \leqslant \frac{59}{87} \\ \frac{1}{29} \leqslant t_7 < \min\left(t_6, \frac{1}{2}(1-t_1-t_2-t_3-t_4-t_5-t_6)\right) \\ \frac{1}{29} \leqslant t_8 < \min\left(t_7, \frac{1}{2}(1-t_1-t_2-t_3-t_4-t_5-t_6-t_7)\right) \\ (t_1, t_2, t_3, t_4, t_5, t_6, t_7, t_8) \text{ cannot be partitioned into } (m, n) \in I }} S\left(\mathcal{A}_{p_1 p_2 p_3 p_4 p_5 p_6 p_7 p_8}, p_8 \right).
\end{equation}
We can also use reversed Buchstab's identity to gain a five-dimensional saving. For the last sum, we cannot decompose it in a straightforward way. However, we can perform a role-reversal to get
\begin{align}
\nonumber & \sum_{\substack{\frac{1}{29} \leqslant t_1 < \frac{1}{2} \\ \frac{1}{29} \leqslant t_2 < \min \left(t_1, \frac{1}{2}(1 - t_1) \right)  \\ (t_1, t_2) \in J_2 \\ \frac{1}{29} \leqslant t_3 < \min \left(t_2, \frac{1}{2}(1 - t_1 - t_2) \right) \\ \frac{1}{29} \leqslant t_4 < \min \left(t_3, \frac{1}{2}(1 - t_1 - t_2 - t_3) \right) \\ (t_1, t_2, t_3, t_4) \text{ cannot be partitioned into } (m, n) \in I \\ t_1 + t_2 + t_3 + 2 t_4 > \frac{59}{87} \\ t_1 + t_2 + t_3 + t_4 \leqslant \frac{59}{87} }} S\left(\mathcal{A}_{p_1 p_2 p_3 p_4}, p_4 \right) \\
\nonumber =&\ \sum_{\substack{\frac{1}{29} \leqslant t_1 < \frac{1}{2} \\ \frac{1}{29} \leqslant t_2 < \min \left(t_1, \frac{1}{2}(1 - t_1) \right)  \\ (t_1, t_2) \in J_2 \\ \frac{1}{29} \leqslant t_3 < \min \left(t_2, \frac{1}{2}(1 - t_1 - t_2) \right) \\ \frac{1}{29} \leqslant t_4 < \min \left(t_3, \frac{1}{2}(1 - t_1 - t_2 - t_3) \right) \\ (t_1, t_2, t_3, t_4) \text{ cannot be partitioned into } (m, n) \in I \\ t_1 + t_2 + t_3 + 2 t_4 > \frac{59}{87} \\ t_1 + t_2 + t_3 + t_4 \leqslant \frac{59}{87} }} S\left(\mathcal{A}_{p_1 p_2 p_3 p_4}, x^{\frac{1}{29}} \right) \\
\nonumber &- \sum_{\substack{\frac{1}{29} \leqslant t_1 < \frac{1}{2} \\ \frac{1}{29} \leqslant t_2 < \min \left(t_1, \frac{1}{2}(1 - t_1) \right)  \\ (t_1, t_2) \in J_2 \\ \frac{1}{29} \leqslant t_3 < \min \left(t_2, \frac{1}{2}(1 - t_1 - t_2) \right) \\ \frac{1}{29} \leqslant t_4 < \min \left(t_3, \frac{1}{2}(1 - t_1 - t_2 - t_3) \right) \\ (t_1, t_2, t_3, t_4) \text{ cannot be partitioned into } (m, n) \in I \\ t_1 + t_2 + t_3 + 2 t_4 > \frac{59}{87} \\ t_1 + t_2 + t_3 + t_4 \leqslant \frac{59}{87} \\ \frac{1}{29} \leqslant t_5 < \min \left(t_4, \frac{1}{2}(1 - t_1 - t_2 - t_3 - t_4) \right) \\ (t_1, t_2, t_3, t_4, t_5) \text{ can be partitioned into } (m, n) \in I }} S\left(\mathcal{A}_{p_1 p_2 p_3 p_4 p_5}, p_{5} \right) \\
\nonumber &- \sum_{\substack{\frac{1}{29} \leqslant t_1 < \frac{1}{2} \\ \frac{1}{29} \leqslant t_2 < \min \left(t_1, \frac{1}{2}(1 - t_1) \right)  \\ (t_1, t_2) \in J_2 \\ \frac{1}{29} \leqslant t_3 < \min \left(t_2, \frac{1}{2}(1 - t_1 - t_2) \right) \\ \frac{1}{29} \leqslant t_4 < \min \left(t_3, \frac{1}{2}(1 - t_1 - t_2 - t_3) \right) \\ (t_1, t_2, t_3, t_4) \text{ cannot be partitioned into } (m, n) \in I \\ t_1 + t_2 + t_3 + 2 t_4 > \frac{59}{87} \\ t_1 + t_2 + t_3 + t_4 \leqslant \frac{59}{87} \\ \frac{1}{29} \leqslant t_5 < \min \left(t_4, \frac{1}{2}(1 - t_1 - t_2 - t_3 - t_4) \right) \\ (t_1, t_2, t_3, t_4, t_5) \text{ cannot be partitioned into } (m, n) \in I }} S\left(\mathcal{A}_{\gamma p_2 p_3 p_4 p_5}, x^{\frac{1}{29}} \right) \\
\nonumber &+ \sum_{\substack{\frac{1}{29} \leqslant t_1 < \frac{1}{2} \\ \frac{1}{29} \leqslant t_2 < \min \left(t_1, \frac{1}{2}(1 - t_1) \right)  \\ (t_1, t_2) \in J_2 \\ \frac{1}{29} \leqslant t_3 < \min \left(t_2, \frac{1}{2}(1 - t_1 - t_2) \right) \\ \frac{1}{29} \leqslant t_4 < \min \left(t_3, \frac{1}{2}(1 - t_1 - t_2 - t_3) \right) \\ (t_1, t_2, t_3, t_4) \text{ cannot be partitioned into } (m, n) \in I \\ t_1 + t_2 + t_3 + 2 t_4 > \frac{59}{87} \\ t_1 + t_2 + t_3 + t_4 \leqslant \frac{59}{87} \\ \frac{1}{29} \leqslant t_5 < \min \left(t_4, \frac{1}{2}(1 - t_1 - t_2 - t_3 - t_4) \right) \\ (t_1, t_2, t_3, t_4, t_5) \text{ cannot be partitioned into } (m, n) \in I \\ \frac{1}{29} \leqslant t_6 < \frac{1}{2}t_1 \\ (1-t_1-t_2-t_3-t_4-t_5, t_2, t_3, t_4, t_5, t_6) \text{ can be partitioned into } (m, n) \in I }} S\left(\mathcal{A}_{\gamma p_2 p_3 p_4 p_5 p_6}, p_6 \right), \\
&+ \sum_{\substack{\frac{1}{29} \leqslant t_1 < \frac{1}{2} \\ \frac{1}{29} \leqslant t_2 < \min \left(t_1, \frac{1}{2}(1 - t_1) \right)  \\ (t_1, t_2) \in J_2 \\ \frac{1}{29} \leqslant t_3 < \min \left(t_2, \frac{1}{2}(1 - t_1 - t_2) \right) \\ \frac{1}{29} \leqslant t_4 < \min \left(t_3, \frac{1}{2}(1 - t_1 - t_2 - t_3) \right) \\ (t_1, t_2, t_3, t_4) \text{ cannot be partitioned into } (m, n) \in I \\ t_1 + t_2 + t_3 + 2 t_4 > \frac{59}{87} \\ t_1 + t_2 + t_3 + t_4 \leqslant \frac{59}{87} \\ \frac{1}{29} \leqslant t_5 < \min \left(t_4, \frac{1}{2}(1 - t_1 - t_2 - t_3 - t_4) \right) \\ (t_1, t_2, t_3, t_4, t_5) \text{ cannot be partitioned into } (m, n) \in I \\ \frac{1}{29} \leqslant t_6 < \frac{1}{2}t_1 \\ (1-t_1-t_2-t_3-t_4-t_5, t_2, t_3, t_4, t_5, t_6) \text{ cannot be partitioned into } (m, n) \in I }} S\left(\mathcal{A}_{\gamma p_2 p_3 p_4 p_5 p_6}, p_6 \right),
\end{align}
where $\gamma \sim x^{1-t_1-t_2-t_3-t_4-t_5}$ and $\left(\gamma , P(p_5)\right)=1$. Since $t_1 + t_2 + t_3 + t_4 \leqslant \frac{59}{87}$ and $(1-t_1-t_2-t_3-t_4-t_5) + t_2 + t_3 + t_4 + t_5 = (1-t_1) \leqslant (1 - \frac{31}{87}) = \frac{56}{87}$, we can give asymptotic formulas for all sums on the right hand side except for the last sum. Note that the last sum in (11) counts numbers with two almost-prime variables, we can make further decompositions on either variable, leading to two eight-dimensional sums
\begin{equation}
\sum_{\substack{\frac{1}{29} \leqslant t_1 < \frac{1}{2} \\ \frac{1}{29} \leqslant t_2 < \min \left(t_1, \frac{1}{2}(1 - t_1) \right)  \\ (t_1, t_2) \in J_2 \\ \frac{1}{29} \leqslant t_3 < \min \left(t_2, \frac{1}{2}(1 - t_1 - t_2) \right) \\ \frac{1}{29} \leqslant t_4 < \min \left(t_3, \frac{1}{2}(1 - t_1 - t_2 - t_3) \right) \\ (t_1, t_2, t_3, t_4) \text{ cannot be partitioned into } (m, n) \in I \\ t_1 + t_2 + t_3 + 2 t_4 > \frac{59}{87} \\ t_1 + t_2 + t_3 + t_4 \leqslant \frac{59}{87} \\ \frac{1}{29} \leqslant t_5 < \min \left(t_4, \frac{1}{2}(1 - t_1 - t_2 - t_3 - t_4) \right) \\ (t_1, t_2, t_3, t_4, t_5) \text{ cannot be partitioned into } (m, n) \in I \\ \frac{1}{29} \leqslant t_6 < \frac{1}{2}t_1 \\ (1-t_1-t_2-t_3-t_4-t_5, t_2, t_3, t_4, t_5, t_6) \text{ cannot be partitioned into } (m, n) \in I \\ (1-t_1-t_2-t_3-t_4-t_5) + t_2 + t_3 + t_4 + t_5 + 2 t_6 \leqslant \frac{59}{87} \\ \frac{1}{29} \leqslant t_7 < \min \left(t_6, \frac{1}{2}(t_1 - t_6) \right) \\ (1-t_1-t_2-t_3-t_4-t_5, t_2, t_3, t_4, t_5, t_6, t_7) \text{ cannot be partitioned into } (m, n) \in I \\ \frac{1}{29} \leqslant t_8 < \min \left(t_7, \frac{1}{2}(t_1 - t_6 - t_7) \right) \\ (1-t_1-t_2-t_3-t_4-t_5, t_2, t_3, t_4, t_5, t_6, t_7, t_8) \text{ cannot be partitioned into } (m, n) \in I }} S\left(\mathcal{A}_{\gamma p_2 p_3 p_4 p_5 p_6 p_7 p_8}, p_8 \right)
\end{equation}
and
\begin{equation}
\sum_{\substack{\frac{1}{29} \leqslant t_1 < \frac{1}{2} \\ \frac{1}{29} \leqslant t_2 < \min \left(t_1, \frac{1}{2}(1 - t_1) \right)  \\ (t_1, t_2) \in J_2 \\ \frac{1}{29} \leqslant t_3 < \min \left(t_2, \frac{1}{2}(1 - t_1 - t_2) \right) \\ \frac{1}{29} \leqslant t_4 < \min \left(t_3, \frac{1}{2}(1 - t_1 - t_2 - t_3) \right) \\ (t_1, t_2, t_3, t_4) \text{ cannot be partitioned into } (m, n) \in I \\ t_1 + t_2 + t_3 + 2 t_4 > \frac{59}{87} \\ t_1 + t_2 + t_3 + t_4 \leqslant \frac{59}{87} \\ \frac{1}{29} \leqslant t_5 < \min \left(t_4, \frac{1}{2}(1 - t_1 - t_2 - t_3 - t_4) \right) \\ (t_1, t_2, t_3, t_4, t_5) \text{ cannot be partitioned into } (m, n) \in I \\ \frac{1}{29} \leqslant t_6 < \frac{1}{2}t_1 \\ (1-t_1-t_2-t_3-t_4-t_5, t_2, t_3, t_4, t_5, t_6) \text{ cannot be partitioned into } (m, n) \in I \\ (1-t_1-t_2-t_3-t_4-t_5) + t_2 + t_3 + t_4 + t_5 + 2 t_6 > \frac{59}{87} \\ (t_1 - t_6) + t_2 + t_3 + t_4 + 2 t_5 + t_6 \leqslant \frac{59}{87} \\ \frac{1}{29} \leqslant t_7 < \min \left(t_5, \frac{1}{2}(1 - t_1 - t_2 - t_3 - t_4 - t_5) \right) \\ (t_1 - t_6, t_2, t_3, t_4, t_5, t_6, t_7) \text{ cannot be partitioned into } (m, n) \in I \\ \frac{1}{29} \leqslant t_8 < \min \left(t_7, \frac{1}{2}(1 - t_1 - t_2 - t_3 - t_4 - t_5 - t_7) \right) \\ (t_1 - t_6, t_2, t_3, t_4, t_5, t_6, t_7, t_8) \text{ cannot be partitioned into } (m, n) \in I }} S\left(\mathcal{A}_{\gamma_1 p_2 p_3 p_4 p_5 p_6 p_7 p_8}, p_8 \right),
\end{equation}
where $\gamma_1 \sim x^{t_1-t_6}$ and $\left(\gamma_1, P(p_6)\right)=1$. We refer the readers to \cite{LiRunbo1215} and \cite{LRB052} for more applications of role-reversals. Combining the cases above we get a loss from $S_{34}$ of
\begin{align}
\nonumber & \left( \int_{\frac{1}{29}}^{\frac{1}{2}} \int_{(t_1, t_2, t_3, t_4) \in J_{341}} \frac{\omega \left(\frac{1 - t_1 - t_2 - t_3 - t_4}{t_4}\right)}{t_1 t_2 t_3 t_4^2} d t_4 d t_3 d t_2 d t_1 \right) \\
\nonumber -& \left( \int_{(t_1, t_2, t_3, t_4, t_5) \in J_{342}} \frac{\omega \left(\frac{1 - t_1 - t_2 - t_3 - t_4 - t_5}{t_5}\right)}{t_1 t_2 t_3 t_4 t_5^2} d t_5 d t_4 d t_3 d t_2 d t_1 \right) \\
\nonumber +& \left( \int_{(t_1, t_2, t_3, t_4, t_5, t_6) \in J_{343}} \frac{\omega_1 \left(\frac{1 - t_1 - t_2 - t_3 - t_4 - t_5 - t_6}{t_6}\right)}{t_1 t_2 t_3 t_4 t_5 t_6^2} d t_6 d t_5 d t_4 d t_3 d t_2 d t_1 \right) \\
\nonumber +& \left( \int_{(t_1, t_2, t_3, t_4, t_5, t_6, t_7, t_8) \in J_{344}} \frac{\max \left(\frac{t_8}{1 - t_1 - t_2 - t_3 - t_4 - t_5 - t_6 - t_7 - t_8}, 0.5672\right)}{t_1 t_2 t_3 t_4 t_5 t_6 t_7 t_8^2} d t_8 d t_7 d t_6 d t_5 d t_4 d t_3 d t_2 d t_1 \right) \\
\nonumber +& \left( \int_{(t_1, t_2, t_3, t_4, t_5, t_6) \in J_{345}} \frac{\omega_1 \left(\frac{1 - t_1 - t_2 - t_3 - t_4 - t_5}{t_5}\right) \omega_1 \left(\frac{t_1 - t_6}{t_6}\right)}{t_2 t_3 t_4 t_5^2 t_6^2} d t_6 d t_5 d t_4 d t_3 d t_2 d t_1 \right) \\
\nonumber +& \left( \int_{(t_1, t_2, t_3, t_4, t_5, t_6, t_7, t_8) \in J_{346}} \frac{\omega_1 \left(\frac{1 - t_1 - t_2 - t_3 - t_4 - t_5}{t_5}\right) \omega_1 \left(\frac{t_1 - t_6 - t_7 - t_8}{t_8}\right)}{t_2 t_3 t_4 t_5^2 t_6 t_7 t_8^2} d t_8 d t_7 d t_6 d t_5 d t_4 d t_3 d t_2 d t_1 \right) \\
\nonumber +& \left( \int_{(t_1, t_2, t_3, t_4, t_5, t_6, t_7, t_8) \in J_{347}} \frac{\omega_1 \left(\frac{1 - t_1 - t_2 - t_3 - t_4 - t_5 - t_7 - t_8}{t_8}\right) \omega_1 \left(\frac{t_1 - t_6}{t_6}\right)}{t_2 t_3 t_4 t_5 t_6^2 t_7 t_8^2} d t_8 d t_7 d t_6 d t_5 d t_4 d t_3 d t_2 d t_1 \right) \\
\nonumber =&\ L_{341} - L_{342} + L_{343} + L_{344} + L_{345} + L_{346} + L_{347}, \\
\nonumber <&\ 0.103669 - 0.03892 + 0.001712 + 0.000001 + 0.007242 + 0.000056 + 0 \\
=&\ 0.07376,
\end{align}
where
\begin{align}
\nonumber J_{341}(t_1, t_2, t_3, t_4) :=&\ \left\{ (t_1, t_2) \in J_{3}, \right. \\
\nonumber & \quad \frac{1}{29} \leqslant t_3 < \min\left(t_2, \frac{1}{2}(1-t_1-t_2)\right), \\ 
\nonumber & \quad \frac{1}{29} \leqslant t_4 < \min\left(t_3, \frac{1}{2}(1-t_1-t_2-t_3) \right), \\
\nonumber & \quad (t_1, t_2, t_3, t_4) \text{ cannot be partitioned into } (m, n) \in I, \\
\nonumber & \quad t_1 + t_2 + t_3 + t_4 > \frac{59}{87}, \\
\nonumber & \left. \quad \frac{1}{29} \leqslant t_1 < \frac{1}{2},\ \frac{1}{29} \leqslant t_2 < \min\left(t_1, \frac{1}{2}(1-t_1) \right) \right\}, \\
\nonumber J_{342}(t_1, t_2, t_3, t_4, t_5) :=&\ \left\{ (t_1, t_2) \in J_{3}, \right. \\
\nonumber & \quad \frac{1}{29} \leqslant t_3 < \min\left(t_2, \frac{1}{2}(1-t_1-t_2)\right), \\ 
\nonumber & \quad \frac{1}{29} \leqslant t_4 < \min\left(t_3, \frac{1}{2}(1-t_1-t_2-t_3) \right), \\
\nonumber & \quad (t_1, t_2, t_3, t_4) \text{ cannot be partitioned into } (m, n) \in I, \\
\nonumber & \quad t_1 + t_2 + t_3 + t_4 > \frac{59}{87}, \\
\nonumber & \quad t_4 < t_5 < \frac{1}{2}(1-t_1-t_2-t_3-t_4), \\
\nonumber & \quad (t_1, t_2, t_3, t_4, t_5) \text{ can be partitioned into } (m, n) \in I, \\
\nonumber & \left. \quad \frac{1}{29} \leqslant t_1 < \frac{1}{2},\ \frac{1}{29} \leqslant t_2 < \min\left(t_1, \frac{1}{2}(1-t_1) \right) \right\}, \\
\nonumber J_{343}(t_1, t_2, t_3, t_4, t_5, t_6) :=&\ \left\{ (t_1, t_2) \in J_{3}, \right. \\
\nonumber & \quad \frac{1}{29} \leqslant t_3 < \min\left(t_2, \frac{1}{2}(1-t_1-t_2)\right), \\ 
\nonumber & \quad \frac{1}{29} \leqslant t_4 < \min\left(t_3, \frac{1}{2}(1-t_1-t_2-t_3) \right), \\
\nonumber & \quad (t_1, t_2, t_3, t_4) \text{ cannot be partitioned into } (m, n) \in I, \\
\nonumber & \quad t_1 + t_2 + t_3 + 2 t_4 \leqslant \frac{59}{87}, \\
\nonumber & \quad \frac{1}{29} \leqslant t_5 < \min\left(t_4, \frac{1}{2}(1-t_1-t_2-t_3-t_4) \right), \\
\nonumber & \quad \frac{1}{29} \leqslant t_6 < \min\left(t_5, \frac{1}{2}(1-t_1-t_2-t_3-t_4-t_5) \right),\\
\nonumber & \quad (t_1, t_2, t_3, t_4, t_5, t_6) \text{ cannot be partitioned into } (m, n) \in I, \\
\nonumber & \quad t_1 + t_2 + t_3 + t_4 + t_5 + 2 t_6 > \frac{59}{87}, \\
\nonumber & \left. \quad \frac{1}{29} \leqslant t_1 < \frac{1}{2},\ \frac{1}{29} \leqslant t_2 < \min\left(t_1, \frac{1}{2}(1-t_1) \right) \right\}, \\
\nonumber J_{344}(t_1, t_2, t_3, t_4, t_5, t_6, t_7, t_8) :=&\ \left\{ (t_1, t_2) \in J_{3}, \right. \\
\nonumber & \quad \frac{1}{29} \leqslant t_3 < \min\left(t_2, \frac{1}{2}(1-t_1-t_2)\right), \\ 
\nonumber & \quad \frac{1}{29} \leqslant t_4 < \min\left(t_3, \frac{1}{2}(1-t_1-t_2-t_3) \right), \\
\nonumber & \quad (t_1, t_2, t_3, t_4) \text{ cannot be partitioned into } (m, n) \in I, \\
\nonumber & \quad t_1 + t_2 + t_3 + 2 t_4 \leqslant \frac{59}{87}, \\
\nonumber & \quad \frac{1}{29} \leqslant t_5 < \min\left(t_4, \frac{1}{2}(1-t_1-t_2-t_3-t_4) \right), \\
\nonumber & \quad \frac{1}{29} \leqslant t_6 < \min\left(t_5, \frac{1}{2}(1-t_1-t_2-t_3-t_4-t_5) \right),\\
\nonumber & \quad (t_1, t_2, t_3, t_4, t_5, t_6) \text{ cannot be partitioned into } (m, n) \in I, \\
\nonumber & \quad t_1 + t_2 + t_3 + t_4 + t_5 + 2 t_6 \leqslant \frac{59}{87}, \\
\nonumber & \quad \frac{1}{29} \leqslant t_7 < \min\left(t_5, \frac{1}{2}(1-t_1-t_2-t_3-t_4-t_5-t_6) \right),\\
\nonumber & \quad \frac{1}{29} \leqslant t_8 < \min\left(t_5, \frac{1}{2}(1-t_1-t_2-t_3-t_4-t_5-t_6-t_7) \right),\\
\nonumber & \quad (t_1, t_2, t_3, t_4, t_5, t_6, t_7, t_8) \text{ cannot be partitioned into } (m, n) \in I, \\
\nonumber & \left. \quad \frac{1}{29} \leqslant t_1 < \frac{1}{2},\ \frac{1}{29} \leqslant t_2 < \min\left(t_1, \frac{1}{2}(1-t_1) \right) \right\}, \\
\nonumber J_{345}(t_1, t_2, t_3, t_4, t_5, t_6) :=&\ \left\{ (t_1, t_2) \in J_{3}, \right. \\
\nonumber & \quad \frac{1}{29} \leqslant t_3 < \min\left(t_2, \frac{1}{2}(1-t_1-t_2)\right), \\ 
\nonumber & \quad \frac{1}{29} \leqslant t_4 < \min\left(t_3, \frac{1}{2}(1-t_1-t_2-t_3) \right), \\
\nonumber & \quad (t_1, t_2, t_3, t_4) \text{ cannot be partitioned into } (m, n) \in I, \\
\nonumber & \quad t_1 + t_2 + t_3 + t_4 \leqslant \frac{59}{87},\ t_1 + t_2 + t_3 + 2 t_4 > \frac{59}{87}, \\
\nonumber & \quad \frac{1}{29} \leqslant t_5 < \min\left(t_4, \frac{1}{2}(1-t_1-t_2-t_3-t_4) \right), \\
\nonumber & \quad (t_1, t_2, t_3, t_4, t_5) \text{ cannot be partitioned into } (m, n) \in I, \\
\nonumber & \quad \frac{1}{29} \leqslant t_6 < \frac{1}{2}t_1,\\
\nonumber & \quad (1-t_1-t_2-t_3-t_4-t_5, t_2, t_3, t_4, t_5, t_6) \\
\nonumber & \qquad \text{ cannot be partitioned into } (m, n) \in I, \\
\nonumber & \left. \quad \frac{1}{29} \leqslant t_1 < \frac{1}{2},\ \frac{1}{29} \leqslant t_2 < \min\left(t_1, \frac{1}{2}(1-t_1) \right) \right\}, \\
\nonumber J_{346}(t_1, t_2, t_3, t_4, t_5, t_6, t_7, t_8) :=&\ \left\{ (t_1, t_2) \in J_{3}, \right. \\
\nonumber & \quad \frac{1}{29} \leqslant t_3 < \min\left(t_2, \frac{1}{2}(1-t_1-t_2)\right), \\ 
\nonumber & \quad \frac{1}{29} \leqslant t_4 < \min\left(t_3, \frac{1}{2}(1-t_1-t_2-t_3) \right), \\
\nonumber & \quad (t_1, t_2, t_3, t_4) \text{ cannot be partitioned into } (m, n) \in I, \\
\nonumber & \quad t_1 + t_2 + t_3 + t_4 \leqslant \frac{59}{87},\ t_1 + t_2 + t_3 + 2 t_4 > \frac{59}{87}, \\
\nonumber & \quad \frac{1}{29} \leqslant t_5 < \min\left(t_4, \frac{1}{2}(1-t_1-t_2-t_3-t_4) \right), \\
\nonumber & \quad (t_1, t_2, t_3, t_4, t_5) \text{ cannot be partitioned into } (m, n) \in I, \\
\nonumber & \quad \frac{1}{29} \leqslant t_6 < \frac{1}{2}t_1,\\
\nonumber & \quad (1-t_1-t_2-t_3-t_4-t_5, t_2, t_3, t_4, t_5, t_6) \\
\nonumber & \qquad \text{ cannot be partitioned into } (m, n) \in I, \\
\nonumber & \quad (1-t_1-t_2-t_3-t_4-t_5) + t_2 + t_3 + t_4 + t_5 + 2 t_6 \leqslant \frac{59}{87}, \\ 
\nonumber & \quad \frac{1}{29} \leqslant t_7 < \min \left(t_6, \frac{1}{2}(t_1 - t_6) \right), \\
\nonumber & \quad (1-t_1-t_2-t_3-t_4-t_5, t_2, t_3, t_4, t_5, t_6, t_7) \\
\nonumber & \qquad \text{ cannot be partitioned into } (m, n) \in I, \\
\nonumber & \quad \frac{1}{29} \leqslant t_8 < \min \left(t_7, \frac{1}{2}(t_1 - t_6 - t_7) \right), \\ 
\nonumber & \quad (1-t_1-t_2-t_3-t_4-t_5, t_2, t_3, t_4, t_5, t_6, t_7, t_8) \\
\nonumber & \qquad \text{ cannot be partitioned into } (m, n) \in I, \\
\nonumber & \left. \quad \frac{1}{29} \leqslant t_1 < \frac{1}{2},\ \frac{1}{29} \leqslant t_2 < \min\left(t_1, \frac{1}{2}(1-t_1) \right) \right\}, \\
\nonumber J_{347}(t_1, t_2, t_3, t_4, t_5, t_6, t_7, t_8) :=&\ \left\{ (t_1, t_2) \in J_{3}, \right. \\
\nonumber & \quad \frac{1}{29} \leqslant t_3 < \min\left(t_2, \frac{1}{2}(1-t_1-t_2)\right), \\ 
\nonumber & \quad \frac{1}{29} \leqslant t_4 < \min\left(t_3, \frac{1}{2}(1-t_1-t_2-t_3) \right), \\
\nonumber & \quad (t_1, t_2, t_3, t_4) \text{ cannot be partitioned into } (m, n) \in I, \\
\nonumber & \quad t_1 + t_2 + t_3 + t_4 \leqslant \frac{59}{87},\ t_1 + t_2 + t_3 + 2 t_4 > \frac{59}{87}, \\
\nonumber & \quad \frac{1}{29} \leqslant t_5 < \min\left(t_4, \frac{1}{2}(1-t_1-t_2-t_3-t_4) \right), \\
\nonumber & \quad (t_1, t_2, t_3, t_4, t_5) \text{ cannot be partitioned into } (m, n) \in I, \\
\nonumber & \quad \frac{1}{29} \leqslant t_6 < \frac{1}{2}t_1,\\
\nonumber & \quad (1-t_1-t_2-t_3-t_4-t_5, t_2, t_3, t_4, t_5, t_6) \\
\nonumber & \qquad \text{ cannot be partitioned into } (m, n) \in I, \\
\nonumber & \quad (1-t_1-t_2-t_3-t_4-t_5) + t_2 + t_3 + t_4 + t_5 + 2 t_6 > \frac{59}{87}, \\ 
\nonumber & \quad (t_1-t_6) + t_2 + t_3 + t_4 + 2 t_5 + t_6 \leqslant \frac{59}{87}, \\ 
\nonumber & \quad \frac{1}{29} \leqslant t_7 < \min\left(t_5, \frac{1}{2}(1-t_1-t_2-t_3-t_4-t_5) \right), \\
\nonumber & \quad (t_1-t_6, t_2, t_3, t_4, t_5, t_6, t_7) \\
\nonumber & \qquad \text{ cannot be partitioned into } (m, n) \in I, \\
\nonumber & \quad \frac{1}{29} \leqslant t_8 < \min\left(t_7, \frac{1}{2}(1-t_1-t_2-t_3-t_4-t_5-t_7) \right), \\
\nonumber & \quad (t_1-t_6, t_2, t_3, t_4, t_5, t_6, t_7, t_8) \\
\nonumber & \qquad \text{ cannot be partitioned into } (m, n) \in I, \\
\nonumber & \left. \quad \frac{1}{29} \leqslant t_1 < \frac{1}{2},\ \frac{1}{29} \leqslant t_2 < \min\left(t_1, \frac{1}{2}(1-t_1) \right) \right\}.
\end{align}

For $S_{35}$ we can also use the devices mentioned earlier, but in this case we will perform a role-reversal on the triple sum if $t_1 + t_2 + t_3 > \frac{59}{87}$. In this way we get a loss of
\begin{align}
\nonumber & \left( \int_{(t_1, t_2, t_3, t_4) \in J_{3501}} \frac{\omega \left(\frac{1 - t_1 - t_2 - t_3 - t_4}{t_4}\right)}{t_1 t_2 t_3 t_4^2} d t_4 d t_3 d t_2 d t_1 \right) \\
\nonumber -& \left( \int_{(t_1, t_2, t_3, t_4, t_5) \in J_{3502}} \frac{\omega \left(\frac{1 - t_1 - t_2 - t_3 - t_4 - t_5}{t_5}\right)}{t_1 t_2 t_3 t_4 t_5^2} d t_5 d t_4 d t_3 d t_2 d t_1 \right) \\
\nonumber +& \left( \int_{(t_1, t_2, t_3, t_4, t_5, t_6) \in J_{3503}} \frac{\omega_1 \left(\frac{1 - t_1 - t_2 - t_3 - t_4 - t_5 - t_6}{t_6}\right)}{t_1 t_2 t_3 t_4 t_5 t_6^2} d t_6 d t_5 d t_4 d t_3 d t_2 d t_1 \right) \\
\nonumber +& \left( \int_{(t_1, t_2, t_3, t_4, t_5, t_6, t_7, t_8) \in J_{3504}} \frac{\max \left(\frac{t_8}{1 - t_1 - t_2 - t_3 - t_4 - t_5 - t_6 - t_7 - t_8}, 0.5672\right)}{t_1 t_2 t_3 t_4 t_5 t_6 t_7 t_8^2} d t_8 d t_7 d t_6 d t_5 d t_4 d t_3 d t_2 d t_1 \right) \\
\nonumber +& \left( \int_{(t_1, t_2, t_3, t_4) \in J_{3505}} \frac{\omega \left(\frac{1 - t_1 - t_2 - t_3}{t_3}\right) \omega \left(\frac{t_1 - t_4}{t_4}\right)}{t_2 t_3^2 t_4^2} d t_4 d t_3 d t_2 d t_1 \right) \\
\nonumber -& \left( \int_{(t_1, t_2, t_3, t_4, t_5) \in J_{3506}} \frac{\omega_0 \left(\frac{1 - t_1 - t_2 - t_3}{t_3}\right) \omega_0 \left(\frac{t_1 - t_4 - t_5}{t_5}\right)}{t_2 t_3^2 t_4 t_5^2} d t_5 d t_4 d t_3 d t_2 d t_1 \right) \\
\nonumber -& \left( \int_{(t_1, t_2, t_3, t_4, t_5) \in J_{3507}} \frac{\omega_0 \left(\frac{1 - t_1 - t_2 - t_3 - t_5}{t_5}\right) \omega_0 \left(\frac{t_1 - t_4}{t_4}\right)}{t_2 t_3 t_4^2 t_5^2} d t_5 d t_4 d t_3 d t_2 d t_1 \right) \\
\nonumber +& \left( \int_{(t_1, t_2, t_3, t_4, t_5, t_6) \in J_{3508}} \frac{\omega_1 \left(\frac{1 - t_1 - t_2 - t_3 - t_6}{t_6}\right) \omega_1 \left(\frac{t_1 - t_4 - t_5}{t_5}\right)}{t_2 t_3 t_4 t_5^2 t_6^2} d t_6 d t_5 d t_4 d t_3 d t_2 d t_1 \right) \\
\nonumber +& \left( \int_{(t_1, t_2, t_3, t_4, t_5, t_6) \in J_{3509}} \frac{\omega_1 \left(\frac{1 - t_1 - t_2 - t_3}{t_3}\right) \omega_1 \left(\frac{t_1 - t_4 - t_5 - t_6}{t_6}\right)}{t_2 t_3^2 t_4 t_5 t_6^2} d t_6 d t_5 d t_4 d t_3 d t_2 d t_1 \right) \\
\nonumber +& \left( \int_{(t_1, t_2, t_3, t_4, t_5, t_6) \in J_{3510}} \frac{\omega_1 \left(\frac{1 - t_1 - t_2 - t_3 - t_5 - t_6}{t_6}\right) \omega_1 \left(\frac{t_1 - t_4}{t_4}\right)}{t_2 t_3 t_4^2 t_5 t_6^2} d t_6 d t_5 d t_4 d t_3 d t_2 d t_1 \right) \\
\nonumber =&\ L_{3501} - L_{3502} + L_{3503} + L_{3504} + L_{3505} - L_{3506} - L_{3507} + L_{3508} + L_{3509} + L_{3510}, \\
\nonumber <&\ 0.079609 - 0.02541 + 0.000001 + 0 + 0.469283 - 0.100821 - 0.124982 + 0.035428 + 0.006114 + 0 \\
=&\ 0.339222,
\end{align}
where
\begin{align}
\nonumber J_{3501}(t_1, t_2, t_3, t_4) :=&\ \left\{ (t_1, t_2) \in J_{4}, \right. \\
\nonumber & \quad \frac{1}{29} \leqslant t_3 < \min\left(t_2, \frac{1}{2}(1-t_1-t_2)\right),\ t_1 + t_2 + t_3 \leqslant \frac{59}{87}, \\ 
\nonumber & \quad \frac{1}{29} \leqslant t_4 < \min\left(t_3, \frac{1}{2}(1-t_1-t_2-t_3) \right), \\
\nonumber & \quad (t_1, t_2, t_3, t_4) \text{ cannot be partitioned into } (m, n) \in I, \\
\nonumber & \quad t_1 + t_2 + t_3 + 2 t_4 > \frac{59}{87}, \\
\nonumber & \left. \quad \frac{1}{29} \leqslant t_1 < \frac{1}{2},\ \frac{1}{29} \leqslant t_2 < \min\left(t_1, \frac{1}{2}(1-t_1) \right) \right\}, \\
\nonumber J_{3502}(t_1, t_2, t_3, t_4, t_5) :=&\ \left\{ (t_1, t_2) \in J_{4}, \right. \\
\nonumber & \quad \frac{1}{29} \leqslant t_3 < \min\left(t_2, \frac{1}{2}(1-t_1-t_2)\right),\ t_1 + t_2 + t_3 \leqslant \frac{59}{87}, \\ 
\nonumber & \quad \frac{1}{29} \leqslant t_4 < \min\left(t_3, \frac{1}{2}(1-t_1-t_2-t_3) \right), \\
\nonumber & \quad (t_1, t_2, t_3, t_4) \text{ cannot be partitioned into } (m, n) \in I, \\
\nonumber & \quad t_1 + t_2 + t_3 + 2 t_4 > \frac{59}{87}, \\
\nonumber & \quad t_4 < t_5 < \frac{1}{2}(1-t_1-t_2-t_3-t_4), \\
\nonumber & \quad (t_1, t_2, t_3, t_4, t_5) \text{ can be partitioned into } (m, n) \in I, \\
\nonumber & \left. \quad \frac{1}{29} \leqslant t_1 < \frac{1}{2},\ \frac{1}{29} \leqslant t_2 < \min\left(t_1, \frac{1}{2}(1-t_1) \right) \right\}, \\
\nonumber J_{3503}(t_1, t_2, t_3, t_4, t_5, t_6) :=&\ \left\{ (t_1, t_2) \in J_{4}, \right. \\
\nonumber & \quad \frac{1}{29} \leqslant t_3 < \min\left(t_2, \frac{1}{2}(1-t_1-t_2)\right),\ t_1 + t_2 + t_3 \leqslant \frac{59}{87}, \\ 
\nonumber & \quad \frac{1}{29} \leqslant t_4 < \min\left(t_3, \frac{1}{2}(1-t_1-t_2-t_3) \right), \\
\nonumber & \quad (t_1, t_2, t_3, t_4) \text{ cannot be partitioned into } (m, n) \in I, \\
\nonumber & \quad t_1 + t_2 + t_3 + 2 t_4 \leqslant \frac{59}{87}, \\
\nonumber & \quad \frac{1}{29} \leqslant t_5 < \min\left(t_4, \frac{1}{2}(1-t_1-t_2-t_3-t_4) \right), \\
\nonumber & \quad \frac{1}{29} \leqslant t_6 < \min\left(t_5, \frac{1}{2}(1-t_1-t_2-t_3-t_4-t_5) \right),\\
\nonumber & \quad (t_1, t_2, t_3, t_4, t_5, t_6) \text{ cannot be partitioned into } (m, n) \in I, \\
\nonumber & \quad t_1 + t_2 + t_3 + t_4 + t_5 + 2 t_6 > \frac{59}{87}, \\
\nonumber & \left. \quad \frac{1}{29} \leqslant t_1 < \frac{1}{2},\ \frac{1}{29} \leqslant t_2 < \min\left(t_1, \frac{1}{2}(1-t_1) \right) \right\}, \\
\nonumber J_{3504}(t_1, t_2, t_3, t_4, t_5, t_6, t_7, t_8) :=&\ \left\{ (t_1, t_2) \in J_{4}, \right. \\
\nonumber & \quad \frac{1}{29} \leqslant t_3 < \min\left(t_2, \frac{1}{2}(1-t_1-t_2)\right),\ t_1 + t_2 + t_3 \leqslant \frac{59}{87}, \\ 
\nonumber & \quad \frac{1}{29} \leqslant t_4 < \min\left(t_3, \frac{1}{2}(1-t_1-t_2-t_3) \right), \\
\nonumber & \quad (t_1, t_2, t_3, t_4) \text{ cannot be partitioned into } (m, n) \in I, \\
\nonumber & \quad t_1 + t_2 + t_3 + 2 t_4 \leqslant \frac{59}{87}, \\
\nonumber & \quad \frac{1}{29} \leqslant t_5 < \min\left(t_4, \frac{1}{2}(1-t_1-t_2-t_3-t_4) \right), \\
\nonumber & \quad \frac{1}{29} \leqslant t_6 < \min\left(t_5, \frac{1}{2}(1-t_1-t_2-t_3-t_4-t_5) \right),\\
\nonumber & \quad (t_1, t_2, t_3, t_4, t_5, t_6) \text{ cannot be partitioned into } (m, n) \in I, \\
\nonumber & \quad t_1 + t_2 + t_3 + t_4 + t_5 + 2 t_6 \leqslant \frac{59}{87}, \\
\nonumber & \quad \frac{1}{29} \leqslant t_7 < \min\left(t_5, \frac{1}{2}(1-t_1-t_2-t_3-t_4-t_5-t_6) \right),\\
\nonumber & \quad \frac{1}{29} \leqslant t_8 < \min\left(t_5, \frac{1}{2}(1-t_1-t_2-t_3-t_4-t_5-t_6-t_7) \right),\\
\nonumber & \quad (t_1, t_2, t_3, t_4, t_5, t_6, t_7, t_8) \text{ cannot be partitioned into } (m, n) \in I, \\
\nonumber & \left. \quad \frac{1}{29} \leqslant t_1 < \frac{1}{2},\ \frac{1}{29} \leqslant t_2 < \min\left(t_1, \frac{1}{2}(1-t_1) \right) \right\}, \\
\nonumber J_{3505}(t_1, t_2, t_3, t_4) :=&\ \left\{ (t_1, t_2) \in J_{4}, \right. \\
\nonumber & \quad \frac{1}{29} \leqslant t_3 < \min\left(t_2, \frac{1}{2}(1-t_1-t_2)\right),\ t_1 + t_2 + t_3 > \frac{59}{87}, \\ 
\nonumber & \quad \frac{1}{29} \leqslant t_4 < \frac{1}{2}t_1, \\
\nonumber & \quad (1-t_1-t_2-t_3, t_2, t_3, t_4) \text{ cannot be partitioned into } (m, n) \in I, \\
\nonumber & \quad (1-t_1-t_2-t_3) + t_2 + t_3 + 2 t_4 > \frac{59}{87}, \\
\nonumber & \quad (t_1 - t_4) + t_2 + 2 t_3 + t_4 > \frac{59}{87}, \\
\nonumber & \left. \quad \frac{1}{29} \leqslant t_1 < \frac{1}{2},\ \frac{1}{29} \leqslant t_2 < \min\left(t_1, \frac{1}{2}(1-t_1) \right) \right\}, \\
\nonumber J_{3506}(t_1, t_2, t_3, t_4, t_5) :=&\ \left\{ (t_1, t_2) \in J_{4}, \right. \\
\nonumber & \quad \frac{1}{29} \leqslant t_3 < \min\left(t_2, \frac{1}{2}(1-t_1-t_2)\right),\ t_1 + t_2 + t_3 > \frac{59}{87}, \\ 
\nonumber & \quad \frac{1}{29} \leqslant t_4 < \frac{1}{2}t_1, \\
\nonumber & \quad (1-t_1-t_2-t_3, t_2, t_3, t_4) \text{ cannot be partitioned into } (m, n) \in I, \\
\nonumber & \quad (1-t_1-t_2-t_3) + t_2 + t_3 + 2 t_4 > \frac{59}{87}, \\
\nonumber & \quad (t_1 - t_4) + t_2 + 2 t_3 + t_4 > \frac{59}{87}, \\
\nonumber & \quad t_4 < t_5 < \frac{1}{2}(t_1-t_4), \\
\nonumber & \quad (1-t_1-t_2-t_3, t_2, t_3, t_4, t_5) \text{ can be partitioned into } (m, n) \in I, \\
\nonumber & \left. \quad \frac{1}{29} \leqslant t_1 < \frac{1}{2},\ \frac{1}{29} \leqslant t_2 < \min\left(t_1, \frac{1}{2}(1-t_1) \right) \right\}, \\
\nonumber J_{3507}(t_1, t_2, t_3, t_4, t_5) :=&\ \left\{ (t_1, t_2) \in J_{4}, \right. \\
\nonumber & \quad \frac{1}{29} \leqslant t_3 < \min\left(t_2, \frac{1}{2}(1-t_1-t_2)\right),\ t_1 + t_2 + t_3 > \frac{59}{87}, \\ 
\nonumber & \quad \frac{1}{29} \leqslant t_4 < \frac{1}{2}t_1, \\
\nonumber & \quad (1-t_1-t_2-t_3, t_2, t_3, t_4) \text{ cannot be partitioned into } (m, n) \in I, \\
\nonumber & \quad (1-t_1-t_2-t_3) + t_2 + t_3 + 2 t_4 > \frac{59}{87}, \\
\nonumber & \quad (t_1 - t_4) + t_2 + 2 t_3 + t_4 > \frac{59}{87}, \\
\nonumber & \quad t_3 < t_5 < \frac{1}{2}(1-t_1-t_2-t_3), \\
\nonumber & \quad (t_1-t_4, t_2, t_3, t_4, t_5) \text{ can be partitioned into } (m, n) \in I, \\
\nonumber & \left. \quad \frac{1}{29} \leqslant t_1 < \frac{1}{2},\ \frac{1}{29} \leqslant t_2 < \min\left(t_1, \frac{1}{2}(1-t_1) \right) \right\}, \\
\nonumber J_{3508}(t_1, t_2, t_3, t_4, t_5, t_6) :=&\ \left\{ (t_1, t_2) \in J_{4}, \right. \\
\nonumber & \quad \frac{1}{29} \leqslant t_3 < \min\left(t_2, \frac{1}{2}(1-t_1-t_2)\right),\ t_1 + t_2 + t_3 > \frac{59}{87}, \\ 
\nonumber & \quad \frac{1}{29} \leqslant t_4 < \frac{1}{2}t_1, \\
\nonumber & \quad (1-t_1-t_2-t_3, t_2, t_3, t_4) \text{ cannot be partitioned into } (m, n) \in I, \\
\nonumber & \quad (1-t_1-t_2-t_3) + t_2 + t_3 + 2 t_4 > \frac{59}{87}, \\
\nonumber & \quad (t_1 - t_4) + t_2 + 2 t_3 + t_4 > \frac{59}{87}, \\
\nonumber & \quad t_4 < t_5 < \frac{1}{2}(t_1-t_4), \\
\nonumber & \quad (1-t_1-t_2-t_3, t_2, t_3, t_4, t_5) \text{ can be partitioned into } (m, n) \in I, \\
\nonumber & \quad t_3 < t_6 < \frac{1}{2}(1-t_1-t_2-t_3), \\
\nonumber & \quad (t_1-t_4, t_2, t_3, t_4, t_6) \text{ can be partitioned into } (m, n) \in I, \\
\nonumber & \left. \quad \frac{1}{29} \leqslant t_1 < \frac{1}{2},\ \frac{1}{29} \leqslant t_2 < \min\left(t_1, \frac{1}{2}(1-t_1) \right) \right\}, \\
\nonumber J_{3509}(t_1, t_2, t_3, t_4, t_5, t_6) :=&\ \left\{ (t_1, t_2) \in J_{4}, \right. \\
\nonumber & \quad \frac{1}{29} \leqslant t_3 < \min\left(t_2, \frac{1}{2}(1-t_1-t_2)\right),\ t_1 + t_2 + t_3 > \frac{59}{87}, \\ 
\nonumber & \quad \frac{1}{29} \leqslant t_4 < \frac{1}{2}t_1, \\
\nonumber & \quad (1-t_1-t_2-t_3, t_2, t_3, t_4) \text{ cannot be partitioned into } (m, n) \in I, \\
\nonumber & \quad (1-t_1-t_2-t_3) + t_2 + t_3 + 2 t_4 \leqslant \frac{59}{87}, \\
\nonumber & \quad \frac{1}{29} \leqslant t_5 < \min\left(t_4, \frac{1}{2}(t_1-t_4)\right), \\
\nonumber & \quad (1-t_1-t_2-t_3, t_2, t_3, t_4, t_5) \text{ cannot be partitioned into } (m, n) \in I, \\
\nonumber & \quad \frac{1}{29} \leqslant t_6 < \min\left(t_5, \frac{1}{2}(t_1-t_4-t_5)\right), \\
\nonumber & \quad (1-t_1-t_2-t_3, t_2, t_3, t_4, t_5, t_6) \text{ cannot be partitioned into } (m, n) \in I, \\
\nonumber & \left. \quad \frac{1}{29} \leqslant t_1 < \frac{1}{2},\ \frac{1}{29} \leqslant t_2 < \min\left(t_1, \frac{1}{2}(1-t_1) \right) \right\}, \\
\nonumber J_{3510}(t_1, t_2, t_3, t_4, t_5, t_6) :=&\ \left\{ (t_1, t_2) \in J_{4}, \right. \\
\nonumber & \quad \frac{1}{29} \leqslant t_3 < \min\left(t_2, \frac{1}{2}(1-t_1-t_2)\right),\ t_1 + t_2 + t_3 > \frac{59}{87}, \\ 
\nonumber & \quad \frac{1}{29} \leqslant t_4 < \frac{1}{2}t_1, \\
\nonumber & \quad (1-t_1-t_2-t_3, t_2, t_3, t_4) \text{ cannot be partitioned into } (m, n) \in I, \\
\nonumber & \quad (1-t_1-t_2-t_3) + t_2 + t_3 + 2 t_4 > \frac{59}{87}, \\
\nonumber & \quad (t_1-t_4) + t_2 + 2 t_3 + t_4 \leqslant \frac{59}{87}, \\
\nonumber & \quad \frac{1}{29} \leqslant t_5 < \min\left(t_3, \frac{1}{2}(1-t_1-t_2-t_3)\right), \\
\nonumber & \quad (t_1-t_4, t_2, t_3, t_4, t_5) \text{ cannot be partitioned into } (m, n) \in I, \\
\nonumber & \quad \frac{1}{29} \leqslant t_6 < \min\left(t_5, \frac{1}{2}(1-t_1-t_2-t_3-t_5)\right), \\
\nonumber & \quad (t_1-t_4, t_2, t_3, t_4, t_5, t_6) \text{ cannot be partitioned into } (m, n) \in I, \\
\nonumber & \left. \quad \frac{1}{29} \leqslant t_1 < \frac{1}{2},\ \frac{1}{29} \leqslant t_2 < \min\left(t_1, \frac{1}{2}(1-t_1) \right) \right\}.
\end{align}

For the remaining $S_{36}$, we choose to discard the whole region giving the loss
\begin{equation}
L_{36} := \int_{\frac{1}{29}}^{\frac{1}{2}} \int_{\frac{1}{29}}^{\min\left(t_1, \frac{1-t_1}{2}\right)} \mathbbm{1}_{(t_1, t_2) \in J_5} \frac{\omega\left(\frac{1 - t_1 - t_2}{t_2}\right)}{t_1 t_2^2} d t_2 d t_1 < 0.093181.
\end{equation}
Note that in this region only products of three primes are counted.

Finally, by (3)--(16), the total loss is less than
\begin{align}
\nonumber & L_{32} + (L_{331} + L_{332} + L_{333} + L_{334}) + (L_{341} + L_{342} + L_{343} + L_{344} + L_{345} + L_{346} + L_{347}) \\
\nonumber & + (L_{3501} + L_{3502} + L_{3503} + L_{3504} + L_{3505} + L_{3506} + L_{3507} + L_{3508} + L_{3509} + L_{3510}) + L_{36} \\
\nonumber <&\ 0.397685 + 0.091383 + 0.07376 + 0.339222 + 0.093181 \\
\nonumber <&\ 0.996
\end{align}
and the proof of Theorem~\ref{t1} is completed.

\section*{Acknowledgements} 
The author would like to thank Professor Chaohua Jia for his encouragement and some helpful discussions.

\bibliographystyle{plain}
\bibliography{bib}

\end{document}